\documentclass[final,a4paper,leqno]{siamltex}

\usepackage[english]{babel}
\usepackage{amsmath,amssymb,amsfonts,graphicx,subfigure}
\usepackage[ruled]{algorithm2e} 
\usepackage{myshortcuts}
\usepackage[utf8]{inputenc}
\usepackage{graphicx}
\usepackage{amsmath,amssymb,amsfonts,mathrsfs}
\usepackage{url}
\usepackage[pdftex, bookmarks=false, colorlinks=true, linkcolor=black, urlcolor=black, citecolor=black, linktoc=section]{hyperref}
\usepackage{epstopdf}

\newtheorem{remark}[theorem]{Remark}

\title{Nonlinearizing two-parameter eigenvalue problems}
\author{Emil Ringh\thanks{
Department of Mathematics, KTH Royal Institute of Technology, SeRC Swedish
e-science research center, Lindstedtsvägen 25, SE-100 44 Stockholm, Sweden, email:
\{eringh,eliasj\}@kth.se
}
\and Elias Jarlebring$^*$}
\date{\today}

\begin{document}

\maketitle
\begin{abstract}
  We investigate a technique to transform a linear
  two-parameter eigenvalue problem, into a nonlinear eigenvalue problem (NEP).
  The transformation stems from an elimination of one of the equations in the two-parameter
  eigenvalue problem, by considering it as a (standard) generalized eigenvalue problem.
  We characterize the equivalence between the original and the nonlinearized problem
  theoretically and show how to use the transformation computationally.
  Special cases of the transformation can be interpreted as a reversed
  companion linearization for polynomial eigenvalue problems,
  as well as a reversed (less known) linearization technique for certain algebraic
  eigenvalue problems with square-root terms.
  Moreover, by exploiting the structure of the NEP we present algorithm specializations for NEP methods, although the technique also allows general solution
  methods for NEPs to be directly applied.
  The nonlinearization is illustrated in examples and simulations,
  with focus on problems where the eliminated equation is of much smaller
  size than the other two-parameter eigenvalue equation.
  This situation arises naturally in domain decomposition techniques.
  A general error analysis is also carried out under the assumption that
  a backward stable eigenvalue solver method is used to solve the
  eliminated problem,
  leading to the conclusion that the error is benign in this situation.
\end{abstract}
\section{Introduction}\label{sec:intro}
This paper concerns the
\emph{two-parameter eigenvalue problem}:
Determine  non-trivial quadruplets
$(\lambda,x,\mu,y)\in\CC\times\CC^n\times\CC\times\CC^m$
such that
\begin{subequations}\label{eq:2ep}
\begin{eqnarray}
  0&=&A_1x+\lambda A_2x+\mu A_3x\label{eq:2ep_a}   \\
  0&=&B_1y+\lambda B_2y+\mu B_3y\label{eq:2ep_b},
\end{eqnarray}
\end{subequations}
where  $A_1,A_2,A_3\in\CC^{n\times n}$, and
$B_1,B_2,B_3\in\CC^{m\times m}$.
We denote the corresponding functions
$A(\lambda,\mu):=A_1+\lambda A_2+\mu A_3$ and
$B(\lambda,\mu):=B_1+\lambda B_2+\mu B_3$.
This problem has been extensively studied in the literature,
see, e.g., the fundamental work of Atkinson \cite{Atkinson:1972:MULTIPARAMETER},
and the summary of
recent developments below.
We assume that
 $m\ll n$ and
that $A_1$, $A_2$ and $A_3$ are large and sparse matrices,
although several theoretical contributions of this paper
are valid without this assumption.

The main idea of our approach can be described as follows.
We view \eqref{eq:2ep_b} as a parameterized generalized
linear eigenvalue problem, where $\lambda$ is the parameter.
Due to perturbation theory for eigenvalue
problems,
there is a family of continuous functions $\{g_i(\lambda)\}$,
defined by the eigenvalues of
\eqref{eq:2ep_b}, where $\mu$ is the eigenvalue,
of a \emph{generalized eigenvalue problem} (GEP).
More formally, for a fixed value of $\lambda$ the
functions $g_i(\lambda)$ and $y_i(\lambda)$ can be defined,
as the solution to
\begin{subequations}\label{eq:nlsys}
\begin{eqnarray}
  0&=&(B_1+\lambda B_2+g_i(\lambda) B_3)y_i(\lambda) \label{eq:nlsys_a} \\
  1&=&c^Ty_i(\lambda), \label{eq:nlsys_b}
\end{eqnarray}
\end{subequations}
for a given vector $c\in\CC^m$.
We explicitly introduced the normalization condition
\eqref{eq:nlsys_b}, to uniquely define a corresponding eigenvector.
The condition \eqref{eq:nlsys_b} is not a restriction
of generality except for the rare situation that
the eigenvector is orthogonal to $c$.
We prefer this condition over the standard Euclidean normalization,
since the right-hand side of \eqref{eq:nlsys_b} is an analytic function.

By insertion of $\mu=g_i(\lambda)$
into \eqref{eq:2ep_a}, we see that a solution to
\eqref{eq:2ep} will satisfy
\begin{equation}  \label{eq:nep}
  M(\lambda)x=(A_1+\lambda A_2+g_i(\lambda) A_3)x=0.
\end{equation}
Note that we have now eliminated $\mu$ and \eqref{eq:2ep_b},
at the cost of the introduction of a nonlinear
function into the eigenvalue problem.
The problem $M(\lambda)x=0$ is called
a \emph{nonlinear eigenvalue problem} (NEP). In our setting it is rather
a family of NEPs, since we have a different nonlinearity for each
function $g_1,\ldots,g_m$.
The study of NEPs is a mature field within numerical
linear algebra, and there are considerable theoretical results,
as well as algorithms and software for NEPs.

The main
contributions of this paper
consist of a theoretical characterization
of the elimination procedure (Section~\ref{sec:theory}),
analysis of structured perturbations
corresponding to the elimination (Section~\ref{sec:conditioning}),
as well as new algorithms for \eqref{eq:2ep} based on
NEP-algorithms (Section~\ref{sec:alg}). We provide
software for the simulations, both for MATLAB and for
Julia \cite{Bezanson2017}.
The Julia software is implemented
using the data structures of the NEP-PACK software
package \cite{Jarlebring:2018:NEPPACK},
including adaption of theory for how to compute derivatives and projections.
This provides new ways to solve
\eqref{eq:2ep}, using the large number of NEP-solvers
available in NEP-PACK. Some contributions are also converse,
i.e.,
we provide insight to NEPs based on the equivalence
with two-parameter eigenvalue problems. For instance,
in Sections~\ref{sec:peplinearization}--\ref{sec:algebraic} we show
how to transform certain NEPs with square-root nonlinearities
to two-parameter eigenvalue problems.
This in turn (using the operator determinants described below)
allows us to transform the problem to a standard generalized eigenvalue
problem, similar to companion linearization techniques for polynomial
and rational eigenvalue problems.

We now summarize the NEP-results relevant for our approach.
For a broad overview see the summary papers
\cite{Ruhe:1973:NLEVP,Mehrmann:2004:NLEVP,Voss2012,Guttel2017},
as well as the benchmark collection \cite{Betcke2010}
and software packages with NEP-solvers
\cite{Roman:2018:SLEPC,Hernandez:2003:SSL,Hernandez:2005:SSF,Jarlebring:2018:NEPPACK}). There is considerable theoretical works
available for the NEP, in particular for
polynomial eigenvalue problems. Techniques to
transform polynomial NEPs to standard eigenvalue
problems (known as linearization)
have been completely characterized in a number of works,
e.g., \cite{Mackey:2006:STRUCTURED,Mackey:2006:VECT} and
\cite{Nakatsukasa:2017:BIVARIATE}.
We relate our approach to this type of linearization
in Section~\ref{sec:peplinearization}.
In our derivation, we make explicit use of the implicit
function theorem applied to the NEP. This
has been done in the context of sensitivity analysis,
leading to eigenvector free formulas for conditioning
\cite{Alam:2019:SENSITIVITY}. There are a number
of algorithms available for NEPs, of which many
seem to be applicable to \eqref{eq:nep}.
More specifically, we characterize the specialization
of residual inverse iteration \cite{Neumaier:1985:RESINV},
which forms the basis of more
recent methods such as the nonlinear Arnoldi
method \cite{Voss:2004:ARNOLDI}.
We also show how the infinite Arnoldi method
\cite{Jarlebring:2012:INFARNOLDI}
can be adapted to \eqref{eq:nep}.

In Section~\ref{sec:dd} we illustrate how two-parameter
eigenvalue problems of this type can arise by the separation of
domains of a partial-differential equation (PDE).
The domains are decoupling  in a way that
the discretization leads to a two-parameter eigenvalue problem.
In this context,
the elimination corresponds to an elimination of
one of the domains.
The elimination of an outer domain, in a way that directly leads to NEPs,
by introduction of artificial boundary
conditions is the origin of several standard
NEPs in the literature, e.g.,  \cite{Tausch2000}
and the electromagnetic cavity model in
\cite{VanBeeumen:2018:GUN2}.

Relevant results for two-parameter eigenvalues can be summarized
as follows. Many results for two-parameter
eigenvalue problems are phrased in the more general setting
of \emph{multiparameter eigenvalue problems}.
There are a number of recent efficient algorithms
available, e.g., based on the Jacobi--Davidson
approach \cite{Hochstenbach:2005:TWOPARAMETER,Hochstenbach:2002:JDDEFTWOPARAM},
including subspace methods in \cite{Hochstenbach:2019:SUBSPACE}.
A number of generalizations
of inverse iteration are derived in \cite{Plestenjak:2016:NUMERICAL}.
Our approach is based on an eigenvalue
parameterization viewpoint. Eigenvalue parameterization
and continuation  techniques
(but with an additional parameter) have been studied,
e.g., in \cite{Plestenjak:2000:CONTINUATION}.

One of the most fundamental properties of two-parameter
eigenvalue problems is the fact that
solutions are given by the solution to a larger
linear (generalized) eigenvalue problem.
This is also often used in the numerical algorithms
mentioned above, and to our knowledge first proposed
as a numerical method in \cite{Slivnik:1986:NUMMETH}.
More precisely, we associate
with \eqref{eq:2ep} the \emph{operator determinants}
\begin{eqnarray}
  \Delta_0&=&B_2\otimes A_3-B_3\otimes A_2\label{eq:Delta0}\\
  \Delta_1&=&B_3\otimes A_1-B_1\otimes A_3\\
  \Delta_2&=&B_1\otimes A_2-B_2\otimes A_1
\end{eqnarray}
where $\otimes$ denotes the Kronecker product.
The solutions to \eqref{eq:2ep} are (under
certain assumptions) equivalent to the solutions to the
two generalized eigenvalue problems
\begin{subequations}\label{eq:deltaeqs}
\begin{eqnarray}
  \Delta_1z&=&\lambda \Delta_0z \label{eq:deltaeqs_a} \\
  \Delta_2z&=&\mu \Delta_0z \label{eq:deltaeqs_b}
\end{eqnarray}
\end{subequations}
where $z=y\otimes x$. In practice, the application
of a general purpose eigenvalue solver on one of the
GEPs
in \eqref{eq:deltaeqs} yields an accurate solution
for small systems.
The equivalence between \eqref{eq:deltaeqs} and
\eqref{eq:2ep} holds under non-singularity assumption;
in particular
the problem is singular if
$A_3$ and $B_3$ both are singular; or
$A_2$ and $B_2$ both are singular.
See \cite{Atkinson:1972:MULTIPARAMETER} for a precise characterization,
and \cite{Kosir:1994:FINITE,Hochstenbach:2003:BACKWARD} for more recent
formulations.

The following matrix is often used in
theory for eigenvalue multiplicity and eigenvalue conditioning,
and will be needed throughout the paper. We denote
\begin{align}\label{eq:cond_mat}
C_0 := \begin{bmatrix}
v^HA_2x & v^HA_3x\\
w^HB_2y  & w^HB_3y
\end{bmatrix},
\end{align}
where $v$ and $w$ are left eigenvectors associated with
\eqref{eq:2ep_a} and \eqref{eq:2ep_b} respectively.
In particular,  for an (algebraically) simple eigenvalue of the two-parameter eigenvalue problem~\eqref{eq:2ep}, the matrix $C_0$ is nonsingular; see \cite[Lemma~3]{Kosir:1994:FINITE}, \cite[Lemma~1.1]{Hochstenbach:2005:TWOPARAMETER}, and \cite[Lemma~1]{Hochstenbach:2003:BACKWARD}.
For a simple eigenvalue, the normwise condition number for the two-parameter eigenvalue problem is expressed as a special induced matrix norm of $C_0^{-1}$, see \cite[Section~4]{Hochstenbach:2003:BACKWARD}.

\section{Nonlinearization}\label{sec:theory}
\subsection{Existence and equivalence}\label{sec:existence_equivalence}
The elimination of the $B$-equation in the
two-parameter eigenvalue problem can be explicitly characterized as we describe
next.
We show how the existence of a nonlinearization can be
explicitly related to the Jordan structure of the  (parametrized) GEP
\begin{equation}  \label{eq:B_eig}
-(B_1+\lambda B_2)y=\mu B_3y,
\end{equation}
which we can also use in practice for the computation of $\mu=g_i(\lambda)$ for a given $\lambda$.
The existence of analytic functions is formalized in the following lemma. The invertibility assumption in the lemma will be further characterized in Theorem~\ref{thm:jordan}.

  \begin{lemma}[Existence of implicit functions]\label{thm:existence}
    Let $\lambda\in\CC$ be given and assume that $(\mu,y)$ is
    such that \eqref{eq:2ep_b} is satisfied with $y$
    normalized as  $c^Ty=1$.
    Moreover, assume that
    \begin{equation}  \label{eq:Jdef}
    J(\lambda,\mu,y):=\begin{bmatrix}
    B(\lambda,\mu)& B_3y \\
    c^T & 0
    \end{bmatrix}
    \end{equation}
    is nonsingular.
  Then, there exist functions
  $g_i:\CC\rightarrow\CC$ and  $y_i:\CC\rightarrow\CC^m$ such that
  \begin{itemize}
  \item $g_i$ and $y_i$ are analytic in $\lambda$,
  \item $g_i$ and $y_i$ satisfy \eqref{eq:nlsys} in a neighborhood of $\lambda$, and
  \item $\mu=g_i(\lambda)$ and $y=y_i(\lambda)$.
  \end{itemize}
\end{lemma}
\begin{proof}
  The proof is based on the complex implicit function theorem. Consider the analytic function $f:\CC^{m+2}\rightarrow\CC^{m+1}$ given by
\begin{equation}\label{eq:fun_impthm}
f(\lambda,\mu,y):=\begin{bmatrix}
B(\lambda,\mu)y\\ c^Ty-1
\end{bmatrix}.
\end{equation}
Then $J$ as in \eqref{eq:Jdef} is the partial Jacobian of $f$ with respect to the variables $y$ and $\mu$, i.e., $J=\partial f/\partial(y, \mu)$. Since $f(\lambda,\mu_i,y_i)=0$ and $J(\lambda,\mu_i,y_i)$ is nonsingular, the existence of the desired functions, analytic in the point $\lambda$, follows from the complex implicit function theorem \cite[Theorem~I.7.6]{Fritzsche:2002:Holomorphic}.
    \end{proof}

Under the same conditions that the implicit functions exist we have the following equivalence between the solutions to the NEP~\eqref{eq:nep} and the solutions to the two-parameter eigenvalue problem~\eqref{eq:2ep}.

  \begin{theorem}[Equivalence]\label{thm:equivalence}
    Suppose the quadruplet $(\lambda,x,\mu,y)\in\CC\times\CC^n\times\CC\times\CC^m$ is such
    that $c^Ty=1$ and $J(\lambda,\mu,y)$ defined in \eqref{eq:Jdef} is nonsingular. Then,
    $(\lambda,x,\mu,y)$ is a solution to \eqref{eq:2ep} if and only
    if $(\lambda,x)$ is a solution to the NEP~\eqref{eq:nep} for one pair of
    functions $(g(\lambda),y(\lambda))=(\mu,y)$
    which satisfies \eqref{eq:nlsys}.
  \end{theorem}

  \begin{proof}
To prove the forward implication direction suppose $(\lambda,x,\mu,y)$ is a solution to \eqref{eq:2ep}.
  From Lemma~\ref{thm:existence},
  there are functions $g(\lambda)$ and $y(\lambda)$  such that $g(\lambda)=\mu$ and $y(\lambda)=y$. Therefore, \eqref{eq:nep} is
  satisfied for that pair $(g(\lambda),y(\lambda))$.

  To prove the backward implication direction suppose $(\lambda,x)$ is a solution to \eqref{eq:nep} for a
  given pair $(g(\lambda),y(\lambda))$. Then
  $(\lambda,x,\mu,y)=(\lambda,x,g(\lambda),y(\lambda))$ is a solution to \eqref{eq:2ep}.
  More precisely, \eqref{eq:2ep_a} is satisfied since \eqref{eq:nep} and
  \eqref{eq:2ep_b} is satisfied due to \eqref{eq:nlsys}.
  \end{proof}

The situation that the Jacobian matrix in \eqref{eq:Jdef} is singular is
a non-generic situation. It turns out, as we show in the following
theorem, that it is singular (essentially) if and only if
the GEP \eqref{eq:B_eig} has a Jordan chain of length two or more.
Therefore, our technique in general works
if a solution to the two parameter eigenvalue problem~\eqref{eq:2ep} corresponds to
a simple eigenvalue $\mu$ of the GEP~\eqref{eq:B_eig}.

\begin{theorem}[Singularity and Jordan structure]\label{thm:jordan}
Let $\lambda\in\CC$ be given and assume that $c$ is not orthogonal to any
eigenvector of the GEP~\eqref{eq:B_eig}. Moreover, assume that $(\mu_i,y_i)$
is a solution to the GEP with $y_i$ normalized as $c^Ty_i=1$.
Then $J(\lambda,\mu_i,y_i)$ is singular if and only if there exists a vector
$u\in\CC^m$ such that
\begin{equation*}
B(\lambda,\mu_i)u + B_3 y_i = 0,
\end{equation*}
i.e., there exists a Jordan chain of length at least two corresponding to the GEP \eqref{eq:B_eig} and eigenpair $(\mu_i,y_i)$.
\end{theorem}
\begin{proof}
We start by proving that singularity implies the existence of a Jordan chain.
Assume that $J(\lambda,\mu_i,y_i)$ is singular. Then there exists a non-trivial vector $\begin{bmatrix}
z &\alpha
\end{bmatrix}^T\in\CC^{m+1}$ such that
$J(\lambda,\mu_i,y_i)\begin{bmatrix}
z &\alpha
\end{bmatrix}^T = 0$.
The first row gives
\begin{equation}\label{eq:jordan_row1}
B(\lambda,\mu_i)z + B_3 y_i \alpha = 0,
\end{equation}
and the second row gives
\begin{equation}\label{eq:jordan_row2}
c^Tz = 0.
\end{equation}
The cases $\alpha=0$ and $\alpha\neq0$ are investigated separately.
Assume that $\alpha=0$, then $z\neq0$ and thus \eqref{eq:jordan_row1} implies that $z$ is an eigenvector to the GEP. However, \eqref{eq:jordan_row2} gives a contradiction. If $\alpha\neq 0$ then \eqref{eq:jordan_row1} gives that there exists a Jordan chain of length at least two, with $u=z/\alpha$. Hence, singularity implies the existence of a Jordan chain.

To prove the converse we assume that there exists a Jordan chain of length at least two. Let $z:=u-(c^Tu) y_i$. Note that from construction of $z$ \eqref{eq:jordan_row2} holds. Moreover, from the Jordan chain we know that \eqref{eq:jordan_row1} holds for the constructed $z$ with $\alpha=1$. Hence, the vector $\begin{bmatrix}
z & 1
\end{bmatrix}^T$ is a non-trivial vector in the null-space of $J(\lambda,\mu_i,y_i)$. Thus the existence of a Jordan chain implies singularity.
\end{proof}

From a practical point of view we know that if we compute simple eigenvalues of the GEP~\eqref{eq:B_eig} such that $c$ is not orthogonal to the corresponding eigenvector, then the Jacobian is nonsingular. Hence, the nonlinearization exists. The result is formalized in the following corollary to Theorem~\ref{thm:jordan}.
\begin{corollary}\label{cor:jordan_simple}
Let $\lambda\in\CC$ be given. Assume that $\mu_i\in\CC$ is a simple eigenvalue of the GEP \eqref{eq:B_eig}, and that $y_i$ is a corresponding right eigenvector normalized as $c^Ty_i=1$. Then $J(\lambda,\mu_i,y_i)$ is nonsingular, where $J$ is defined in \eqref{eq:Jdef}.
\end{corollary}

Many algorithms for NEPs depend on analyticity in a target domain. From the above theory we directly conclude the following result for the convergence radius of the implicitly constructed analytic function.
\begin{corollary}[Convergence radius]\label{cor:conv_rad}
  Let $\lambda\in\CC$ be given and assume that $(\mu_i,y_i)$ is
  a solution to the GEP \eqref{eq:B_eig} with
  $y_i$ normalized as $c^Ty_i=1$. Moreover, assume
  that $J(\lambda,\mu_i,y_i)$ is invertible, where
  $J$ is defined in \eqref{eq:Jdef}.
  The functions $g_i$ and $y_i$
  in Theorem~\ref{thm:existence} can be chosen such that
  they are analytic in the open disk with radius $r$ centered in $\lambda$,
  where $r$ is defined by
  \[
  r=\min\big\{|\lambda-s|:\textrm{ such that }
    g_i(s) \textrm{ is a  double eigenvalue of the GEP }\eqref{eq:B_eig}\big\}.
  \]
\end{corollary}
\begin{proof} By the definition of $r$, Theorem~\ref{thm:jordan} ensures that $J(\lambda,g_i(\lambda),y_i)$ is nonsingular in all points in the open disk. Application of Theorem~\ref{thm:existence} to all those points establishes the result.
  \end{proof}

As discussed above,
the choice of solution to
the GEP~\eqref{eq:B_eig}
corresponds to the choice of function $g_i(\lambda)$. We note that from Corollary~\ref{cor:jordan_simple} it is clear that the existence of the nonlinearization only relies on that the chosen eigenvalue, of the GEP, is simple. Similarly from  Corollary~\ref{cor:conv_rad} it is clear that the convergence radius is dependent only on the specific function $g_i$ in consideration. Hence, the existence, and the convergence radius, of the NEP~\eqref{eq:nep} only depends on the behavior of $g_i$ and not the complete eigenstructure of the GEP.

\subsection{Nonlinearizations leading to quadratic eigenvalue problems}\label{sec:peplinearization}
We first illustrate the theory in the previous section with an
implicitly defined function which can be derived explicitly.
Consider the two-parameter
  eigenvalue problem
\begin{subequations}\label{eq:2ep_qep}
\begin{eqnarray}
  0&=&A_1x+\lambda A_2x+\mu A_3x  \label{eq:2ep_qep_a} \\
  0&=&\left(\begin{bmatrix}0&0\\0&-1\end{bmatrix}
    +\lambda
\begin{bmatrix}0&1\\1&0\end{bmatrix}
  +\mu
\begin{bmatrix}-1&0\\0&0\end{bmatrix}\right) y, \label{eq:2ep_qep_b}
\end{eqnarray}
\end{subequations}
for general matrices $A_1$, $A_2$ and $A_3$. The second row
in \eqref{eq:2ep_qep_b} implies that the elements in
the vector $y^T=\begin{bmatrix} y_1 & y_2\end{bmatrix}$ are related by $y_2=\lambda y_1$.
The first row in \eqref{eq:2ep_qep_b} becomes
$\lambda^2y_1-\mu y_1=0$.
Hence, since $y_1\neq 0$, we have $\mu=\lambda^2$ and
\eqref{eq:2ep_qep_a} becomes
\begin{equation}  \label{eq:qep}
  0=A_1x+\lambda A_2x+\lambda^2 A_3x.
\end{equation}
  This problem is commonly known as the quadratic eigenvalue problem,
  which has been extensively studied in the literature \cite{Tisseur:2001:QUADRATIC}. The
  example shows that the two-parameter eigenvalue problem
  \eqref{eq:2ep_qep} can be nonlinearized to a quadratic eigenvalue
  problem.
  Moreover, the determinant operator equation \eqref{eq:deltaeqs_a} leads to the
  equation
  \[
  \begin{bmatrix}
    -A_1&0\\
    0 & A_3
  \end{bmatrix}z=\lambda\begin{bmatrix}
  A_2&A_3\\
  A_3& 0
  \end{bmatrix}z,
  \]
  which is a particular companion linearization of \eqref{eq:qep}. (It is in fact
  a symmetry preserving linearization \cite[Section~3.4]{Tisseur:2001:QUADRATIC}.)
  Many of the linearizations of polynomial eigenvalue problems given in
  \cite{Mackey:2006:VECT} can be obtained in a similar fashion.
  Since, the second equation \eqref{eq:2ep_b} can be expressed
  as $\det(B(\lambda,\mu))=0$, which is a bi-variate polynomial,
  this example is consistent
  with the bivariate viewpoint of companion linearizations
  in \cite{Nakatsukasa:2017:BIVARIATE}. Some higher-degree polynomials can be constructed
  analogously to above, e.g., the polynomial eigenvalue problem
  $A_1+\lambda A_2+\lambda^mA_3$.
  However, the general higher-degree polynomial eigenvalue problem does not seem
  to fit into the class of two-parameter eigenvalue problems.

   \begin{figure}[h]
     \begin{center}
      \subfigure[$\re(g_+(\lambda))$]{\scalebox{1}{\includegraphics{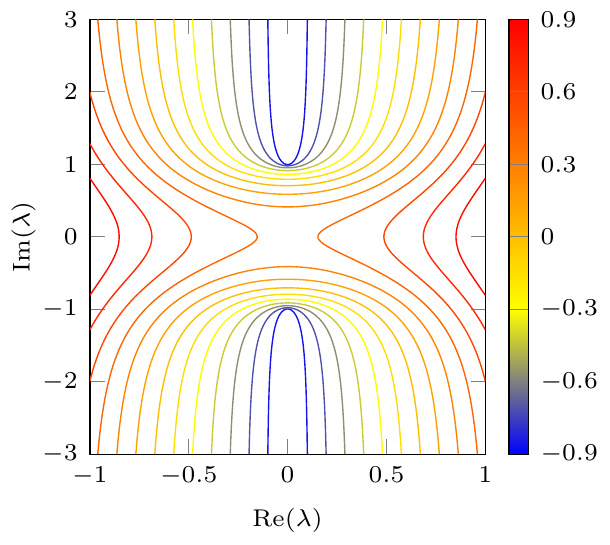}}}
      \subfigure[$\im(g_+(\lambda))$]{\scalebox{1}{\includegraphics{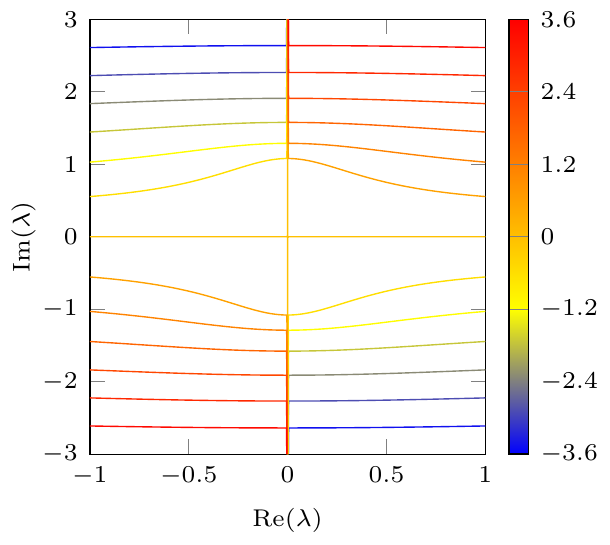}}}
    \caption{The square root nonlinearity illustrated in the example in Section~\ref{sec:algebraic}, with
$a=3$,
$b=2$,
$c=-1$,
$d=-2$,
$e=2$ and
      $f=1$. We observe a square root singularity at $\lambda=\pm\sqrt{-3/2}$ which are
      the roots of $p(\lambda)$.
      \label{fig:two_by_two}
    }
     \end{center}
   \end{figure}

\subsection{Nonlinearization leading to algebraic functions}\label{sec:algebraic}
The previous example can be modified in a way that
it leads to algebraic functions, which is also the generic situation.
  Nontrivial solutions to~\eqref{eq:2ep_b} satisfy $\det(B(\lambda,\mu))=0$,
  which is a bivariate polynomial. Therefore,
  the functions $g_i(\lambda)$ are roots of a polynomial,
  where the coefficients are polynomials in $\lambda$, i.e.,
  $g_i$ are algebraic functions. The generic situation
  can be seen from the case where $m=2$:
\begin{subequations}\label{eq:sqrt_MEP}
\begin{eqnarray}
    0&=&(A_1+\lambda A_2+\mu A_3)x  \label{eq:sqrt_MEP_a} \\
0&=&\left(\begin{bmatrix}
  a&b\\c&d
\end{bmatrix}
+\lambda
\begin{bmatrix}
  0&e\\f&0
\end{bmatrix}
+\mu
\begin{bmatrix}
  1&0\\0&1
\end{bmatrix}\right)y. \label{eq:sqrt_MEP_b}
 \end{eqnarray}
\end{subequations}
We obtain that
$\mu$ is the root of a polynomial, where the coefficients depend
on $\lambda$, i.e.,
\[
0=(\mu+a)(\mu+d)-(c+\lambda f)(b+\lambda e).
\]
The explicit solutions to this quadratic equation are given by
\[
\mu=g_{\pm}(\lambda)=-\frac{a+d}{2}\pm\sqrt{\frac{(a+d)^2}{4}-ad+(b+\lambda e)(c+\lambda f)}.
\]
We see by insertion of $\mu=g_{\pm}$  into \eqref{eq:sqrt_MEP_a} that
the nonlinearization of \eqref{eq:sqrt_MEP} is a NEP with an algebraic nonlinearity.
The function $g_+$ is illustrated in Figure~\ref{fig:two_by_two}.

Several general conclusions can be made from this example. Note that
the variables $a,b,c,d,e,f$ can be used for fitting of any function $\sqrt{p(\lambda)}$
where $p$ is a polynomial of degree two.
Therefore, we can now reverse the nonlinearization,
and for the trivial case $a=d=0$ we directly obtain the following
characterization.
\begin{lemma}[Two-parameterization of an algebraic NEP]\label{lem:sqrt}
  Suppose $p(\lambda)=(b+\lambda e)(c+\lambda f)$ is given,
  and let $a=d=0$.
  If $(\lambda,x)$ is a solution to the NEP
\begin{equation}\label{eq:NEP_algebraic}
(A_1+\lambda A_2+\sqrt{p(\lambda)}A_3)x,
\end{equation}
then  $(\lambda,x,\mu,y)$ satisfies the two-parameter eigenvalue problem equation \eqref{eq:sqrt_MEP} with
 $\mu:=\sqrt{p(\lambda)}$ and
\[
y:=\begin{bmatrix}
\sqrt{b+\lambda e}\\
-\sqrt{c+\lambda f}
\end{bmatrix}.
\]
  \end{lemma}

A further consequence of the lemma is that problems
of the type \eqref{eq:NEP_algebraic} can be
linearized to a standard GEP using
the determinant operators \eqref{eq:deltaeqs}.
More precisely, the combination of Lemma~\ref{lem:sqrt} and \eqref{eq:deltaeqs}
shows that \eqref{eq:NEP_algebraic} can be solved by computing
solutions to
\begin{align*}
\begin{bmatrix}
A_1 & -bA_3\\
-cA_3 & A_1
\end{bmatrix}z
=
\lambda
\begin{bmatrix}
-A_2 & e A_3\\
fA_3 & -A_2
\end{bmatrix}z.
\end{align*}
The fact that algebraic NEPs can be linearized was already pointed out in the
conference presentation \cite{Shao:2018:ALGNEP}, for
a specific case using techniques not involving two-parameter eigenvalue
problems.

Also note that the functions $g_i(\lambda)$ have branch-point singularities.
This is the generic situation and we can therefore never expect that
the nonlinearizations are entire functions in general. The singularities restrict
the performance of many methods, as we will see in the simulations.
The implications of singularities in practice is
well-known in quantum chemistry, where
parameterized eigenvalue problems is a fundamental
tool and the singularities are referred to as
intruder states \cite[Chapter~14]{Helgaker:2000:ELECTRONIC}.
In that context, methods for computing the closest singularity (which limits
the performance of the method) are given in \cite{Jarlebring2011a,Kvaal:2011:PHYSPAPER}.

\section{Algorithm specializations}\label{sec:alg}

\subsection{Derivative based algorithms}\label{sec:der}
Many NEP-algorithms are based on derivatives of $M$. We will now
illustrate how to efficiently and reliably access the derivatives
of the NEP stemming from a nonlinearization of a two-parameter eigenvalue
problem. As a representative first situation
we consider the augmented Newton method, see
\cite{Ruhe:1973:NLEVP,Unger:1950:NICHTLINEARE}.
It can be derived by an elimination
of the correction equation in Newtons method, and
leads to separate eigenvalue and eigenvector update
formulas expressed as
\begin{subequations}\label{eq:augnewton}
\begin{eqnarray}
 x_{k+1}&=&\alpha_kM(\lambda_k)^{-1}M'(\lambda_k)x_{k}\\
 \alpha_k^{-1}&=&d^TM(\lambda_k)^{-1}M'(\lambda_k)x_{k}
\end{eqnarray}
\end{subequations}
and $\lambda_{k+1}=\lambda_k-\alpha_k$, where $d\in\CC^{n}$ is a normalization vector.
In an implementation, one takes advantage of the fact that the same linear system
appears twice, and only needs to be computed once.
The iteration has appeared in many variations with different names, e.g.,
inverse iteration \cite{Spence:2005:PHOTONIC} and Newtons method \cite{unger2013convergence}.

In order to apply \eqref{eq:augnewton} we clearly need the derivative of $M$
defined in \eqref{eq:nep}, which can be
obtained directly if we can compute the derivative of the implicitly defined function $g_i$.
Note that the functions $g_i(\lambda)$ (as well as the auxiliary vector  $y_i(\lambda)$)
can be evaluated by solving the GEP~\eqref{eq:B_eig},
and normalizing according to $c^Ty_i=1$.
Since the functions are analytic in general,
their respective derivatives exist. They can be computed according to the following result,
which gives a recursion that can compute the $k$th derivative by solving $k$
linear systems of dimension $(m+1)\times(m+1)$.
The adaption of the theorem and \eqref{eq:augnewton}  into an algorithm results in  Algorithm~\ref{alg:augnewton}.
\begin{theorem}[Explicit recursive form for derivatives]\label{thm:der}
  Let $\lambda\in\CC$ be given and assume that $(\mu_i,y_i)$
  is a solution to the GEP \eqref{eq:B_eig} with
  $y_i$ normalized as $c^Ty_i=1$. Moreover, assume
  that $J(\lambda,\mu_i,y_i)$ is invertible, where
  $J$ is defined in \eqref{eq:Jdef}. Let
  $g_i$ and $y_i$ be the functions defined
  in Lemma~\ref{thm:existence},
  then the $k$th derivative, $k=1,2,\dots$, of $g_i$ and $y_i$ are given by
  \begin{equation}  \label{eq:J_der}
    \begin{bmatrix}
      y_i^{(k)}(\lambda)\\
      g_i^{(k)}(\lambda)
    \end{bmatrix}
    =J(\lambda,\mu_i,y_i)^{-1}\begin{bmatrix}
    -b_k\\
    0
    \end{bmatrix},
  \end{equation}
    where
    \[
 b_k=B_2y_i^{(k-1)}(\lambda)+\sum_{j=1}^{k-1}{k \choose j} g_i^{(k-j)}(\lambda)B_3y_i^{(j)}(\lambda).
    \]
\end{theorem}
\begin{proof}
We again consider the analytic function $f$ given by \eqref{eq:fun_impthm}.
By Lemma~\ref{thm:existence} we know that $g_i$ and $y_i$ are analytic around $\lambda$, and that $f(\lambda,g_i(\lambda),y_i(\lambda))=0$ in a neighborhood of $\lambda$. Taking the $k$th implicit derivative with respect to $\lambda$ gives
\begin{equation*}
0
=
\frac{d^k}{d\lambda^k}
\begin{bmatrix}
B_1y_i(\lambda)\\ c^Ty_i(\lambda)-1
\end{bmatrix}
+
\frac{d^k}{d\lambda^k}
\begin{bmatrix}
\lambda B_2 y_i(\lambda)\\0
\end{bmatrix}
+
\frac{d^k}{d\lambda^k}
\begin{bmatrix}
g_i(\lambda)B_3 y_i(\lambda)\\ 0
\end{bmatrix}.
\end{equation*}
The two first terms are found directly as
\begin{equation*}
\frac{d^k}{d\lambda^k}
\begin{bmatrix}
B_1y_i(\lambda)\\ c^Ty_i(\lambda)-1
\end{bmatrix}
=
\begin{bmatrix}
B_1y_i^{(k)}(\lambda)\\ c^Ty_i^{(k)}(\lambda)
\end{bmatrix},
\end{equation*}
and
\begin{equation*}
\frac{d^k}{d\lambda^k}
\begin{bmatrix}
\lambda B_2 y_i(\lambda)\\0
\end{bmatrix}
=
\begin{bmatrix}
\lambda B_2 y_i^{(k)}(\lambda)\\0
\end{bmatrix}
+
\begin{bmatrix}
B_2 y_i^{(k-1)}(\lambda)\\0
\end{bmatrix}.
\end{equation*}
The third term can be calculated, by using Leibniz derivation rule for products, to be
\begin{align*}
&\frac{d^k}{d\lambda^k}
\begin{bmatrix}
g_i(\lambda)B_3 y_i(\lambda)\\ 0
\end{bmatrix}
=\\
&\begin{bmatrix}
\sum_{j=1}^{k-1} {k \choose j} g_i^{(k-j)}(\lambda) B_3 y_i^{(j)}(\lambda)\\0
\end{bmatrix}
+
\begin{bmatrix}
g_i^{(k)}(\lambda) B_3 y_i(\lambda)\\0
\end{bmatrix}
+
\begin{bmatrix}
 g_i(\lambda) B_3 y_i^{(k)}(\lambda)\\0
\end{bmatrix}
.
\end{align*}
We emphasize the recursion: All derivatives up to order $k-1$ can be considered known since these do not depend on the higher derivatives. Collecting the known terms in the right-hand-side gives the result.
\end{proof}
\begin{remark}\label{rem:g_der}
As a special case of Theorem~\ref{thm:der}, for $k=1$, we find that
$g_i'(\lambda) = -\frac{w_i^H B_2 y_i}{w_i^H B_3 y_i}$
where $w_i$ is the corresponding left eigenvector to the eigenpair $(\mu_i,y_i)$.
It follows from multiplying the first block-row of equation system~\eqref{eq:J_der} from the left with $w_i^H$.
The result is a special case of well known perturbation analysis for generalized eigenvalue problems \cite[Theorem~2.5]{Higham:1998:Structured}. In our case $g'(\lambda)$ is the perturbation of the eigenvalue $\mu$ with respect to $\lambda$ in the GEP~\eqref{eq:B_eig}. More precisely, a perturbation of the matrix $-(B_1+\lambda B_2)$ with the structured perturbation $\varepsilon B_2$.

Specifically, the closed form of $g_i'(\lambda)$ means that the derivative of the NEP~\eqref{eq:nep} can be written in closed form, as
\begin{align*}
M'(\lambda) = A_2 - \frac{w_i^HB_2y_i}{w_i^HB_3y_i}A_3.
\end{align*}
For methods only requiring the first derivative of $M(\lambda)$, the above expression can be used instead of \eqref{eq:J_der}. However, that requires the computations of the left eigenvector of the GEP.
We will need the expression for theoretical purposes in Section~\ref{sec:conditioning}.
\end{remark}

\begin{algorithm}
\caption{Augmented Newton method for nonlinearized two-parameter eigenvalue problem\label{alg:augnewton}}
\SetKwInOut{Input}{input}\SetKwInOut{Output}{output}
\Input{Starting values $\lambda_0\in\CC$ and $x_0\in\CC^n$}
\Output{Approximations of eigenpairs of \eqref{eq:nep}}
\BlankLine
\nl \For{$k=1,\ldots,$}{
  \nl Compute $g_i(\lambda_k):=\mu$ from the GEP \eqref{eq:B_eig} with $c$-normalized eigenvector $y\in\CC^m$  \label{step:gi}\\
\nl \If{$\|A(\lambda_k,\mu)x_k\|\le$TOL}{\nl \bf break}
\nl Compute $g_i'(\lambda_k)$ by computing $b_1=B_2y$ and solving the linear system of equations \eqref{eq:J_der} \label{step:gip}\\
\nl Compute $u=M(\lambda_k)^{-1}M'(\lambda_k)x_k$ by using the results in Steps~\ref{step:gi}--\ref{step:gip}\\
\nl Compute $\alpha_k=(d^Tu)^{-1}$\\
\nl Compute $x_{k+1}=\alpha_k y$\\
\nl Compute $\lambda_{k+1}=\lambda_{k}-\alpha_k^{-1}$\\
}
\end{algorithm}

The family of methods in \cite{Jarlebring:2012:INFARNOLDI,Jarlebring:2017:TIAR,Mele2018} (flavors of the infinite Arnoldi method)  also requires derivative information.
These methods require computation of quantities such as
\begin{eqnarray*}
  z_0&=&M(\sigma)^{-1}(M'(\sigma)x_1+\cdots+M^{(p)}(\sigma)x_p)   \\
        &=&M(\sigma)^{-1}(A_1x_1+A_2\sum_{j=1}^pg^{(j)}(\sigma)x_j),
\end{eqnarray*}
where $x_1,\dots,x_p$ are given vectors.
The computation requires higher derivatives of $g_i$. However, $\sigma$ is unchanged throughout
the iteration and therefore the matrix in the linear system for derivative computation \eqref{eq:J_der}
is unchanged. Hence, all needed derivatives can be computed by solving
an additional linear system. If  $m \ll n$, this will in general
not be computationally demanding.
We also note that these fixed-shift methods choose a branch $g_i$ in the initial solution of the GEP~\eqref{eq:B_eig}, and then stay on that branch. Hence, convergence properties will depend on the convergence radius of that function $g_i$, as mentioned in the end of Section~\ref{sec:existence_equivalence}.

\subsection{Projection methods}\label{sec:proj}

Many NEP-algorithms require the computation of a projected problem
\begin{equation}\label{eq:proj_eq}
W^TM(\lambda)Vz=0
\end{equation}
where $V,W\in\CC^{n\times p}$ are orthogonal matrices.
The problem \eqref{eq:proj_eq} is again a NEP, but of smaller
size. This can be viewed as a Petrov--Galerkin
projection of the spaces spanned by the columns of $V$ and $W$.
The projection is sometimes called subspace acceleration (or
the nonlinear Rayleigh--Ritz procedure), since it
is often used to improve properties
of a more basic algorithm, e.g., the nonlinear
Arnoldi method \cite{Voss:2004:ARNOLDI},
Jacobi--Davidson methods \cite{Effenberger2013,Betcke:2004:JD},
block preconditioned harmonic projection methods \cite{Xue:2018:BLOCK}, the infinite Lanczos method \cite{mele:2018:infinite},
and many more.

In order to give access to these methods, we need
to provide a way to solve \eqref{eq:proj_eq} for our
nonlinearized problem.
Fortunately, the projected problem stemming
from the nonlinearized two-parameter eigenvalue problem,
i.e.,
\begin{equation} \label{eq:proj}
 (W^TA_1V+\lambda W^TA_2V+g_i(\lambda)W^TA_3V)z=0,
\end{equation}
has a structure which suggests straightforward
methods for the projected problem.
This is because the projected NEP has the same
structure as the nonlinearized two-parameter eigenvalue problem,
and can therefore be lifted back to a two-parameter eigenvalue problem,
but now of much smaller size.
We can then use general methods for two-parameter eigenvalue
problems.
This is directly  observed from the
fact that \eqref{eq:proj}
is the nonlinearization of a two-parameter eigenvalue
problem with projected $A$-matrices. It is made more
precise in the following result.

\begin{corollary}[Projected nonlinearized problem]\label{thm:proj}
Suppose the quadruplet $(\lambda,z,\mu,y)\in\CC\times\CC^p\times\times \CC\times\CC^m$ is
such that $c^Ty=1$ and suppose $J(\lambda,\mu,y)$ with $J$
as defined in \eqref{eq:Jdef} is nonsingular. Then,
$(\lambda,z,\mu,y)$ is a solution to the
  the two-parameter eigenvalue problem
  \begin{subequations}\label{eq:proj2EP}
  \begin{eqnarray}
    0&=&W^TA_1Vz+\lambda W^TA_2Vz+\mu W^TA_3Vz \label{eq:proj2EP_a}  \\
    0&=&B_1y+\lambda B_2y+\mu B_3y \label{eq:proj2EP_b}
  \end{eqnarray}
  \end{subequations}
if and only if $(\lambda,z)$ is a solution
to \eqref{eq:proj} for one pair of
    functions $(g(\lambda),y(\lambda))=(\mu,y)$
    which satisfies \eqref{eq:nlsys}.
  \end{corollary}
\begin{proof}
This follows directly from the application of Theorem~\ref{thm:equivalence}
on the projected problem \eqref{eq:proj2EP} and the NEP \eqref{eq:proj}.
\end{proof}

If the projection space is small
$p\ll n$, and $m\ll n$, we may even solve the two-parameter
eigenvalue problem using the operator determinant
eigenvalue equations \eqref{eq:deltaeqs} or \cite[Algorithm~2.3]{Hochstenbach:2005:TWOPARAMETER}.

The situation $p=1$ implies that the projected problem
is a scalar problem, and reduces to the so-called Rayleigh functional. There
are several methods based on the Rayleigh functional, e.g., residual inverse iteration \cite{Neumaier:1985:RESINV}, and variational principle based approaches
such as \cite{Szyld:2016:PRECOND} and references therein.
The fact that the projected problem is scalar and linear allows us
to eliminate $\mu$, and we find that $\lambda$ is a solution to the generalized eigenvalue problem.
The following corollary specifies the formulas more precisely, and the adaption
of the result into the residual inverse iteration is
given in Algorithm~\ref{alg:resinv}.

\begin{corollary}
The solution to the projected NEP \eqref{eq:proj} with $p=1$ is given
by $\lambda$, $\mu\in\CC$ and $y\in\CC^m$, where $(\lambda,y)$ is
a solution to the GEP
\begin{equation}  \label{eq:scalar_proj}
     ((w^TA_3v)B_1-(w^TA_1v) B_3)y=\lambda( (w^TA_2v)B_3 - (w^TA_3v)B_2)y,
\end{equation}
and $\mu$ is given by
\begin{equation}\label{eq:scalar_proj_mu}
\mu=-\frac{w^TA_1v+\lambda w^TA_2v}{w^TA_3v}.
\end{equation}
\end{corollary}
\begin{proof}
  This is derived from a  special case of Corollary~\ref{thm:proj} where $p=1$. The relation
  \eqref{eq:proj2EP_a} with $W=w$ and $V=v$ can be solved for $\mu$
  resulting in the relation \eqref{eq:scalar_proj_mu}. By inserting this relation into
  \eqref{eq:proj2EP_b} we obtain the GEP \eqref{eq:scalar_proj}.
  \end{proof}
\begin{algorithm}
\caption{Resinv for nonlinearized two-parameter eigenvalue problem\label{alg:resinv}}
\SetKwInOut{Input}{input}\SetKwInOut{Output}{output}
\Input{Approximate eigenvector $x_0\in\CC^n$, shift $\sigma\in\CC$, right Rayleigh functional vector $w\in\CC^n$}
\Output{Approximations of eigenpairs of \eqref{eq:nep}}
\BlankLine
\nl Compute $M(\sigma)$ and factorize\label{step:factorize}\\
\While{not converged}{
\nl Compute $\lambda_{k+1}=\lambda$ by solving the GEP \eqref{eq:scalar_proj} for $v=x_k$ . \\
\nl Compute $\mu$ from \eqref{eq:scalar_proj_mu} with $v=x_k$\\
\nl Compute $z:=M(\lambda_{k+1})x_k=A_0x_k+\lambda_{k+1}A_1x_k+\mu A_2x_k$.\\
\nl Compute correction $u_{k+1}=x_k-M(\sigma)^{-1}z$ using the factorization
computed in Step~\ref{step:factorize}\\
\nl Normalize $x_{k+1}=u_{k+1}/\|u_{k+1}\|$
}
\end{algorithm}

\section{Conditioning and accuracy}\label{sec:conditioning}
In order to characterize when the elimination procedure works well, we now analyze how
the technique behaves subject to perturbations. As a consequence of this we can directly conclude how backward stable computation of $g$ influences the accuracy (Section~\ref{sec:backward}).\footnote{For notational convenience the $i$ index on $g_i$ is dropped in this section.}

\subsection{Conditioning as a nonlinear eigenvalue problem}
Standard results for the condition number of NEPs can be used to analyze perturbations with respect to the $A$-matrices. More precisely, if we define
\begin{align*}
\kappa_A(\lambda) := \limsup_{\varepsilon\rightarrow 0}
\left\{\frac{|\Delta\lambda|}{\varepsilon}: \|\Delta A_j\|\leq\varepsilon\alpha_j,j=1,2,3 \right\},
\end{align*}
where $\alpha_j$ are scalars for $j=1,2,3$, and $\Delta\lambda$ is such that
\begin{align}\label{eq:nep_pert}
0 = (A_1+\Delta A_1+(\lambda+\Delta\lambda)(A_2+\Delta A_2)+g(\lambda+\Delta\lambda)(A_3+\Delta A_3))(x+\Delta x),
\end{align}
then we know (see, e.g., \cite{Alam:2019:SENSITIVITY}) that
\begin{align}\label{eq:kappa_A}
\kappa_A(\lambda) =  \|v\|\|x\| \frac{\alpha_1 + |\lambda|\alpha_2  + |g(\lambda)| \alpha_3}{|v^H M'(\lambda)x|},
\end{align}
where $v,x$ are the corresponding left and right eigenvectors.
In the following we will establish how this formula is modified when we also consider perturbations in the $B$-matrices.
Note that this implies that the function $g$ is also perturbed and we cannot directly use the standard result.
We therefore define the condition number
\begin{align*}
\kappa(\lambda) := \limsup_{\varepsilon\rightarrow 0}
\left\{\frac{|\Delta\lambda|}{\varepsilon}: \|\Delta A_j\|\leq\varepsilon\alpha_j,j=1,2,3 \text{ and }\|\Delta B_j\|\leq\varepsilon\beta_j,j=1,2,3\right\},
\end{align*}
where $\beta_j$ are scalars for $j=1,2,3$, and $\Delta\lambda$ fulfills \eqref{eq:nep_pert} but with a perturbed $g$, i.e., $\mu+\Delta\mu = g(\lambda+\Delta\lambda)$, such that
\begin{subequations}\label{eq:gep_pert}
\begin{align}
  0&= \left(B_1+\Delta B_1 + (\lambda+\Delta \lambda)(B_2+\Delta B_2) + (\mu+\Delta \mu)(B_3+\Delta B_3)\right)(y+\Delta y) \label{eq:gep_pert_a}\\
  1&=c^T (y+\Delta y). \label{eq:gep_pert_b}
\end{align}
\end{subequations}
The definitions can be used both for \emph{absolute} and \emph{relative condition} numbers by setting $\alpha_j=\beta_j = 1$ or $\alpha_j = \|A_j\|$, $\beta_j = \|B_{j}\|$ for $j=1,2,3$ respectively.

As an intermediate step we first consider the perturbation of $\mu$ subject to perturbations in the $B$-matrices and fixed perturbations in $\lambda$ by analyzing
\begin{align*}
\kappa_g(\lambda):=\limsup_{\varepsilon\rightarrow 0}
\left\{\frac{|\Delta \mu|}{\varepsilon}: |\Delta\lambda|\leq\varepsilon\gamma \text{ and } \|\Delta B_j\|\leq\varepsilon\beta_j, j=1,2,3\right\},
\end{align*}
where $\gamma$ is a scalar, and $\Delta\mu$ satisfies \eqref{eq:gep_pert} for a given $\lambda$ . The following result shows that $\kappa_g$ can be expressed as a sum of perturbations associated with the $B$-matrices and perturbations associated with $\lambda$.

\begin{lemma}\label{lemma:g_cond}
Let $\lambda\in\CC$ be given and suppose $\mu=g(\lambda)$ is a simple eigenvalue of the GEP~\eqref{eq:B_eig} with $w$ and $y$ being corresponding left and right eigenvectors. Then,
\begin{align*}
\kappa_g(\lambda) = \kappa_{g,B}(\lambda) + \kappa_{g,\lambda}(\lambda),
\end{align*}
where
\begin{align*}
\kappa_{g,B}(\lambda) = \|w\|\|y\|\frac{ \beta_1  + |\lambda| \beta_2 +  |g(\lambda)| \beta_3}{|w^H B_3 y|}
\qquad\text{ and }\qquad
\kappa_{g,\lambda}(\lambda) =
\gamma \frac{|w^H B_2 y|}{|w^H B_3 y|}.
\end{align*}
  \end{lemma}
\begin{proof}
Since $\mu$ is a simple eigenvalue of the GEP~\eqref{eq:B_eig},
the eigenvalue and eigenvector are analytic, and therefore $\Delta y = \OOO(\varepsilon)$ when all the perturbations are $\OOO(\varepsilon)$. By collecting all the higher order terms the perturbed GEP~\eqref{eq:gep_pert_a} can thus be written as
\begin{align*}
\left(\Delta B_1 + \lambda\Delta B_2 + \Delta\lambda B_2 + \mu\Delta B_3 + \Delta \mu B_3\right)y
+
B(\lambda,\mu)\Delta y
=
\OOO(\varepsilon^2).
\end{align*}
Multiplying with $w^H$ from the left, solving for $\Delta\mu$, and dividing with $\varepsilon$ gives that
\begin{align}\label{eq:delta_mu}
\frac{\Delta \mu}{\varepsilon}
= - \frac{ w^H \Delta  B_1 y + \lambda w^H\Delta B_2 y + \Delta\lambda w^H B_2 y +  \mu w^H \Delta B_3 y}{\varepsilon w^H B_3 y}
+ \OOO(\varepsilon).
\end{align}
An upper bound is thus found as
\begin{align*}
\frac{\Delta \mu}{\varepsilon}
\leq \|w\|\|y\|\frac{ \beta_1  + |\lambda| \beta_2 +  |\mu| \beta_3}{|w^H B_3 y|}
+
\eta\frac{|w^H B_2 y|}{|w^H B_3 y|}
+ \OOO(\varepsilon).
\end{align*}
It remains to show that the bound can be attained.
This follows from considering $\hat B = wy^H/\|w\|\|y\|$, and inserting
\begin{align*}
&\Delta B_1 = -\varepsilon \beta_1\hat B
&&\Delta B_2 = -\varepsilon \frac{\overline{\lambda}}{|\lambda|} \beta_2\hat B
\\
&\Delta B_3 = -\varepsilon \frac{\overline{g(\lambda)}}{|g(\lambda)|} \beta_3\hat B
&&\ \ \Delta \lambda = -\varepsilon \frac{\overline{w^H B_2 y}}{|w^H B_2 y|}\frac{|w^H B_3 y|}{\overline{w^H B_3 y}} \eta,
\end{align*}
into \eqref{eq:delta_mu}.
\end{proof}

Using the intermediate result we can now show that the condition number $\kappa(\lambda)$ is the sum of the standard condition number of NEPs and a term representing perturbations in $g$ generated by perturbations in the $B$-matrices, i.e., $\kappa_{g,B}(\lambda)$.

\begin{theorem}\label{thm:cond}
Let $\lambda$ be a simple eigenvalue of the NEP~\eqref{eq:nep} with $v$ and $x$ being corresponding left and right eigenvectors. Moreover, for this $\lambda$, suppose $\mu=g(\lambda)$ is a simple eigenvalue of the GEP~\eqref{eq:B_eig} with $w$ and $y$ being corresponding left and right eigenvectors. Then,
\begin{align*}
\kappa(\lambda) =  \kappa_A(\lambda)
+
\kappa_{g,B}(\lambda)
\frac{|v^H A_3 x|}{|v^H M'(\lambda)x|},
\end{align*}
where $\kappa_A(\lambda)$ is given by \eqref{eq:kappa_A}.
\end{theorem}

\begin{proof}
Recall the assumptions that the NEP~\eqref{eq:nep}, i.e., $M$, is analytic, that $\lambda$ is a simple eigenvalue of the NEP, and that $\mu$ is a simple eigenvalue of the GEP~\eqref{eq:B_eig}. Hence,
the eigenvalues and eigenvectors are analytic, and therefore $\Delta x = \OOO(\varepsilon)$ when all the perturbations are $\OOO(\varepsilon)$. By using that $g(\lambda+\Delta\lambda)=g(\lambda)+\Delta\mu$ and collecting all the higher order terms, the perturbed NEP~\eqref{eq:nep_pert} can therefore
be written as
\begin{align*}
(\Delta A_1+\lambda\Delta A_2+\Delta\lambda A_2+g(\lambda)\Delta A_3+\Delta\mu A_3)x
+
M(\lambda)\Delta x
=
\OOO(\varepsilon^2).
\end{align*}
Multiplying with $v^H$ from the left, expanding $\Delta\mu$ according to~\eqref{eq:delta_mu}, solving for $\Delta\lambda$, and dividing with $\varepsilon$, gives that
\begin{align}\label{eq:delta_lambda}
\frac{\Delta\lambda}{\varepsilon}  = - \frac{v^H\Delta A_1x +\lambda v^H\Delta A_2x + g(\lambda) v^H\Delta A_3 x
+\theta_{g,B}(\lambda)
v^H A_3 x}{\varepsilon v^H\left(A_2 - \frac{w^HB_2y}{w^HB_3y}A_3\right)x}
+
\OOO(\varepsilon),
\end{align}
where
$\theta_{g,B}(\lambda) := -
( w^H \Delta  B_1 y + \lambda w^H\Delta B_2 y +  g(\lambda) w^H \Delta B_3 y)/(w^H B_3 y)$.
Based on Remark~\ref{rem:g_der} we observe that the denominator of \eqref{eq:delta_lambda} is equal to $\varepsilon v^H M'(\lambda)x$. An upper bound is is therefore
\begin{align*}
\frac{\Delta\lambda}{\varepsilon}  \leq  \|v\|\|x\| \frac{\alpha_1 + |\lambda|\alpha_2  + |g(\lambda)| \alpha_3}{|v^H M'(\lambda)x|}
+
\kappa_{g,B}(\lambda)
\frac{
|v^H A_3 x|}{|v^H M'(\lambda)x|}
+
\OOO(\varepsilon).
\end{align*}
It remains to show that the bound can be attained. Similar to
the proof of Lemma~\ref{lemma:g_cond}, this
follows from considering $\hat B = wy^H/\|w\|\|y\|$ and $\hat A = vx^H/\|v\|\|x\|$, and inserting
\begin{align*}
&\Delta B_1 = \varepsilon \beta_1\hat B
\qquad
&&\Delta B_2 = \varepsilon \frac{\overline{\lambda}}{|\lambda|} \beta_2\hat B
\qquad
&&\Delta B_3 = \varepsilon \frac{\overline{g(\lambda)}}{|g(\lambda)|} \beta_3\hat B\\
&\Delta A_1 = -\varepsilon \alpha_1\hat A
\qquad
&&\Delta A_2 = -\varepsilon \frac{\overline{\lambda}}{|\lambda|} \alpha_2\hat A
\qquad
&&\Delta A_3 = -\varepsilon \frac{\overline{g(\lambda)}}{|g(\lambda)|} \alpha_3\hat A,
\end{align*}
into \eqref{eq:delta_lambda}.
\end{proof}

\subsection{Backward stable computation of $g$}\label{sec:backward}
The nonlinearization is based on solving a GEP to evaluate the function $g(\lambda)$.
We analyze the effects on the accuracy in the computed $\lambda$ when the GEP is solved numerically
with a backward stable method. The analysis assumes the two triplets $(\lambda,x,v)\in\CC\times\CC^n\times\CC^n$ and $(\mu,y,w)\in\CC\times\CC^m\times\CC^m$ are such that $\lambda$ is a simple eigenvalue of the NEP~\eqref{eq:nep}, $\mu$ is a simple eigenvalue of the GEP~\eqref{eq:B_eig}, and $v,w$ and $x,y$ are corresponding left and right eigenvectors respectively.

From the assumption that the GEP~\eqref{eq:B_eig} is solved by a backward stable method we know that $\mu$ can be characterized as the exact solution to a nearby problem. More precisely, $\mu$ solves
\begin{align*}
(C_1+\Delta C_1) y = \mu (C_2+\Delta C_2) y,
\end{align*}
where $C_1 = -(B_1+\lambda B_2)$, $C_2 = B_3$, with perturbations, $\Delta C_1$ and $\Delta C_2$, that are proportional to the errors in our GEP solver.
Specifically, there are non-negative $\beta_1,\beta_3\in\RR$ such that $\|\Delta C_1\| = \beta_1\varepsilon$ and $\|\Delta C_2\| = \beta_3\varepsilon$.
Thus, the perturbation in $g$ is precisely captured by $\kappa_{g,B}(\lambda)$ from Lemma~\ref{lemma:g_cond}, with $\beta_2=0$ and $\beta_1$ and $\beta_3$ given above, i.e., by the specific choice of GEP solver.
Hence, by application of Theorem~\ref{thm:cond} with $\alpha_j=0$ for $j=1,2,3$ we can conclude that the forward error in $\lambda$, induced by the inexact but backward stable computation of $g(\lambda)$ is bounded by
\begin{align}\label{eq:Dl_1}
 |\Delta\lambda| \leq \|w\|\|y\|\frac{ \beta_1  +  |g(\lambda)| \beta_3}{|w^H B_3 y|} \frac{
|v^H A_3 x|}{|v^H M'(\lambda)x|}\varepsilon + \OOO(\varepsilon^2).
\end{align}
Without loss of generality we now assume that $\|x\|=\|v\|=\|y\|=\|w\|=1$.

The upper bound~\eqref{eq:Dl_1} is
related to the condition number for multiparameter eigenvalue problems as follows.
As mentioned in the introduction, the condition number
for the two-parameter eigenvalue problem
can be directly expressed with
the inverse of $C_0$ defined in \eqref{eq:cond_mat}.
First note that our assumptions imply that $C_0$ is invertible.
\begin{lemma}\label{lemma:C_0}
Under the conditions of Theorem~\ref{thm:cond} the matrix $C_0$ is nonsingular, where $C_0$ is defined in \eqref{eq:cond_mat}.
\end{lemma}
\begin{proof}
By using the expression for $M'(\lambda)$ from Remark~\ref{rem:g_der} we thus have
\begin{align}\label{eq:det_C0}
(w^HB_3y)(v^HM'(\lambda)x)
= (v^HA_2x)(w^HB_3y)-(v^HA_3x)(w^HB_2y)
=\det(C_0).
\end{align}
Since the eigenvalues $\lambda$ and $\mu$ are simple we know that $w^HB_3y\neq 0$, and that $v^HM'(\lambda)x \neq 0$. Hence, $\det(C_0)\neq0$.
\end{proof}

From \eqref{eq:det_C0} we can conclude that the bound \eqref{eq:Dl_1} on $|\Delta\lambda|$ can be written as
\begin{align}\label{eq:Dl_2}
 |\Delta\lambda| \leq  \left(\beta_1  +  |g(\lambda)| \beta_3\right) \frac{|v^H A_3 x|}{|\det(C_0)|}\varepsilon + \OOO(\varepsilon^2).
\end{align}
Moreover, for a nonsingular $C_0$ it is shown in \cite[Theorem~6]{Hochstenbach:2003:BACKWARD} that the condition number of the two-parameter eigenvalue is
\begin{align*}
K
=
\|C_0^{-1}\|_\theta,
\end{align*}
where the $\theta$-norm, i.e., $\|\cdot\|_\theta$, is an induced norm defined in~\cite[Equation~(5)]{Hochstenbach:2003:BACKWARD}.
In our case we can explicitly bound the condition number
by using bounds following directly from the definition of the $\theta$-norm:
\begin{multline*}
 \|C_0^{-1}\|_\theta
=
\frac{1}{|\det(C_0)|}\left\|
\begin{bmatrix}
w^HB_3y & -v^HA_3x\\
-w^HB_2y  & v^HA_2x
\end{bmatrix}
\right\|_\theta
\geq\\
\frac{1}{|\det(C_0)|}\left\|
\begin{bmatrix}
0 & -v^HA_3x\\
0  & 0
\end{bmatrix}
\right\|_\theta
=
\frac{|v^HA_3x||\theta_2|}{|\det(C_0)|}.
\end{multline*}
The parameter $\theta_2$ is the second component of the $\theta$-vector used in the definition of the $\theta$-norm. Hence, the bound in \eqref{eq:Dl_2} can be further bounded by
\begin{align}\label{eq:Dl_3}
|\Delta\lambda| &\leq  K  \frac{\beta_1  +  |g(\lambda)| \beta_3}{|\theta_2|}\varepsilon + \OOO(\varepsilon^2).
\end{align}
The typical choices of $\theta$ corresponding to the absolute respectively relative condition number of the two-parameter eigenvalue problem
are $|\theta_2| = 1+|\lambda|+|g(\lambda)|$ and $|\theta_2| = \|B_1\|+|\lambda|\|B_2\|+|g(\lambda)|\|B_3\|$.
From the bounds in~\eqref{eq:Dl_3} we therefore conclude: The error generated by a backward stable method is benign for well conditioned two-parameter eigenvalue problems.

\section{Simulations}\label{sec:sim}

\subsection{Random example}\label{sec:random}
We generate an example similar to the example in \cite{Hochstenbach:2005:TWOPARAMETER},
but with $m\ll n$.
More precisely, we let
\[
A_i=\alpha_i V_{A_i}F_iU_{A_i}, \;\;
B_i=\beta_iV_{B_i}G_iU_{B_i},\;\; i=1,2,3
\]
where $n=5000$ and $m=20$. The matrices
$V_{A_i}$, $U_{A_i}$, $V_{B_i}$, $U_{B_i}$ have randomly normal distributed elements
and $F_i$,  $G_i$ are diagonal matrices with randomly normal distributed diagonal elements. The scalars $\alpha_i$ and $\beta_i$ were selected
such that the eigenvalues closest to the origin were of order of
magnitude one in modulus
($\alpha_1=\beta_1=1$, $\alpha_2=\beta_2=1/500$, $\alpha_3=\beta_3=1/50$).
The simulations were carried out using the Julia language \cite{Bezanson2017} (version 1.1.0), but
implementations of the algorithms are available online for both
Julia and MATLAB.\footnote{The matrices and the simulations are provided online for
reproducibility: \url{http://www.math.kth.se/~eliasj/src/nonlinearization}}

\begin{figure}[h]
  \begin{center}
    \subfigure[Real]{\scalebox{1.0}{\includegraphics{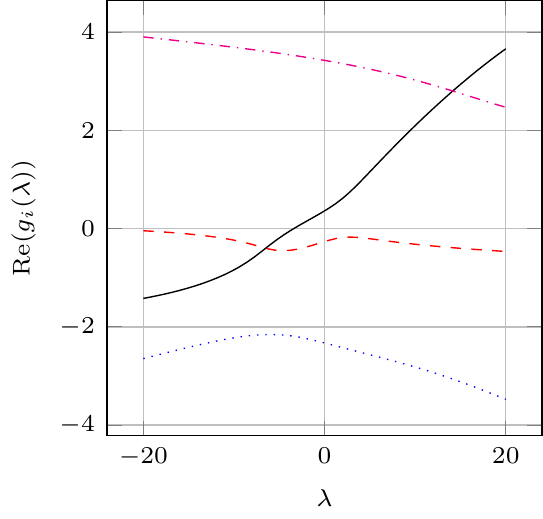}}}\;\;\;\;\;\;%
    \subfigure[Imag]{\scalebox{1.0}{\includegraphics{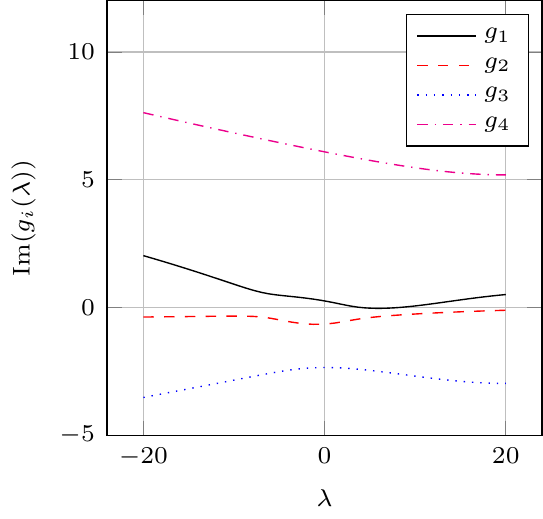}}}
    \caption{The functions $g_i(\lambda)$, $i=1,\ldots,4$ closest to the origin, for $\lambda\in [-20,20]$. All functions are  analytic in the considered interval.
      \label{fig:random:g_of_lambda}
    }
  \end{center}
\end{figure}

\begin{figure}[h]
  \begin{center}
    \subfigure[Solutions $\lambda$]{\scalebox{1.0}{\includegraphics{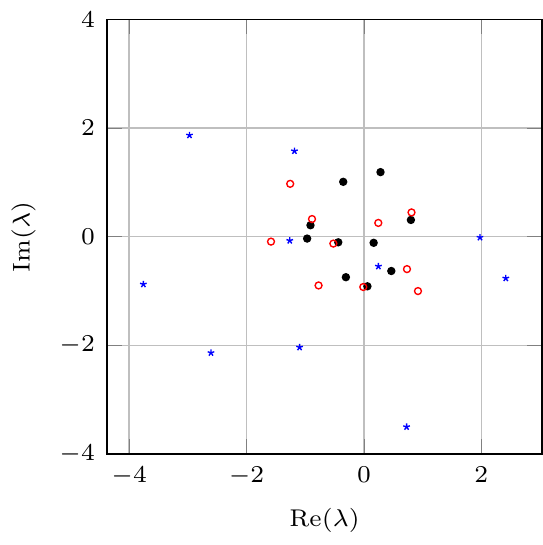}}}\;\;\;\;%
    \subfigure[Solutions $\mu$]{\scalebox{1.0}{\includegraphics{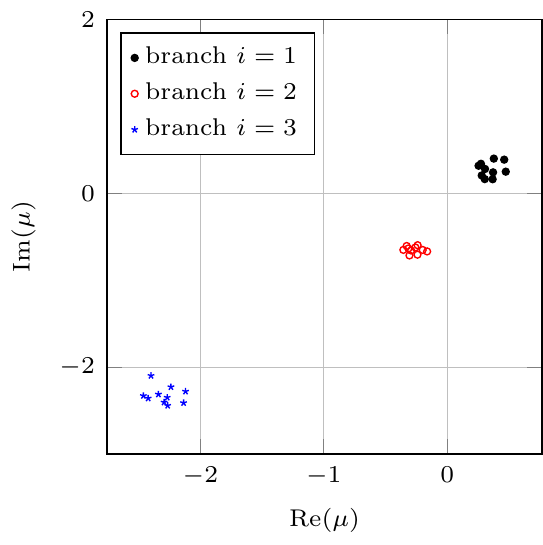}}}%
    \caption{Solutions corresponding to $g_i$ where $i=1,2,3$.
      \label{fig:random:solutions}
    }
  \end{center}
\end{figure}

Since $m=20$, we in general obtain $20$ different functions $g_1,\ldots,g_{20}$,
which we order by magnitude in the origin,
each corresponding to a different NEP.
Some of the nonlinear functions $g_i$ are visualized in
Figure~\ref{fig:random:g_of_lambda}. The solutions closest to the origin,
for the NEPs corresponding to the functions $g_1,g_2,g_3$,
are given in
Figure~\ref{fig:random:solutions}.

Note that this problem is of such size that
the naive approach with operator determinants \eqref{eq:deltaeqs}
is not feasible, since we cannot even store them in memory on the
computers we use for the simulations\footnote{The simulations were carried out using Ubuntu Linux, Intel(R) Core(TM) i7-8550U CPU @ 1.80GHz, 16 GB of RAM}.

We illustrate our algorithms and compare with several other single-vector
state-of-the-art algorithms in \cite{Plestenjak:2016:NUMERICAL}.
As starting values we use $\lambda_0=0.15+0.1i$ and $\mu_0=35+0.25i$,
and a starting vector
with an elementwise absolute error (from a nearby solution) less than $0.05$.
The iteration history of Algorithm~\ref{alg:augnewton}
is given in Figure~\ref{fig:random:augnewton}.
We observe an asymptotic fast convergence for Algorithm~\ref{alg:augnewton}, which is expected
since the solution point is analytic and simple.
The error is measured at Step~3 in
Algorithm~\ref{alg:augnewton} which implies that by construction,
the error in the $B$-equation is (numerically) zero. This
is a property of the elimination in our approach.
We compare (with the same starting values) with
the inverse iteration Newton approach
proposed in \cite{Plestenjak:2016:NUMERICAL}.
Note that this method is designed for more general problems,
and not specifically our
situation where $m\ll n$ and also multiparameter nonlinear problems.
Since \cite[InvIter]{Plestenjak:2016:NUMERICAL} requires several linear
solves per iteration, our algorithm is faster in this case.
The comparison between the two algorithms as a function of
iteration is inconclusive, as can be seen in Figure~\ref{fig:random:augnewton}a.

\begin{figure}[h]
  \begin{center}
       \subfigure[Residual norm vs iteration]{\includegraphics{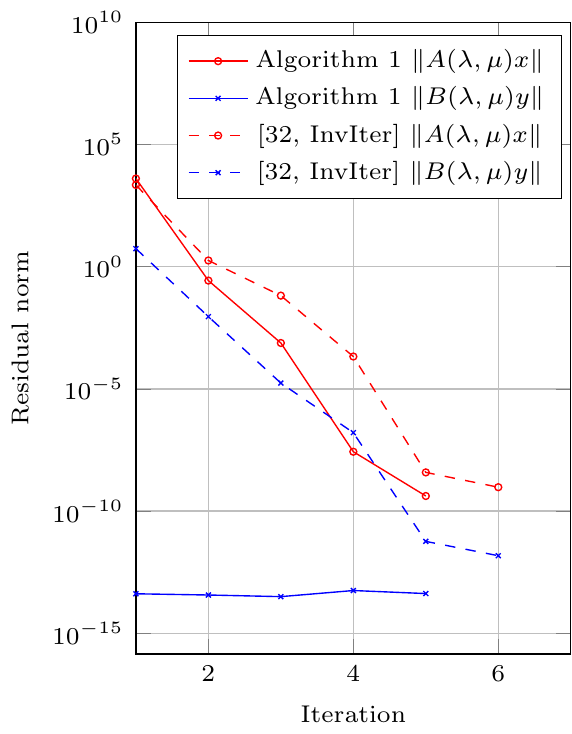}}\;\;\;\;%
   \subfigure[Residual norm vs CPU-time]{\includegraphics{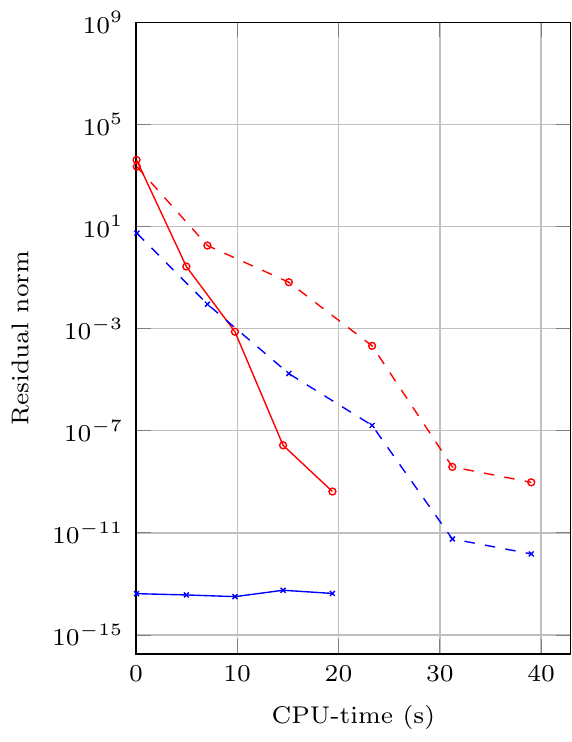}}
   \caption{Visualization of the convergence of Algorithm~\ref{alg:augnewton}
and \cite[Algorithm 1 (InvIter)]{Plestenjak:2016:NUMERICAL} for the problem
     in Section~\ref{sec:random}.
      \label{fig:random:augnewton}
    }
  \end{center}
\end{figure}

The convergence of our adaption of residual inverse iteration (Algorithm~\ref{alg:resinv})
initiated in the same way (except the starting vector is chosen as a vector
of ones) is illustrated in Figure~\ref{fig:random:resinv}.
We clearly see the expected linear convergence, since
it is equivalent to residual inverse iteration for NEPs and the convergence
theory
in \cite[Section~3--4]{Neumaier:1985:RESINV} is directly applicable.
We compare with a proposed generalization of residual inverse iteration
\cite[InvIter]{Plestenjak:2016:NUMERICAL},  again noting that it has a
much wider applicability domain than our approach. In this case, our method
has a smaller convergence factor, intuitively motivated by the
fact that we solve the $B$-equation exactly.

\begin{figure}[h]
  \begin{center}
       \subfigure[Residual norm vs iteration]{\includegraphics{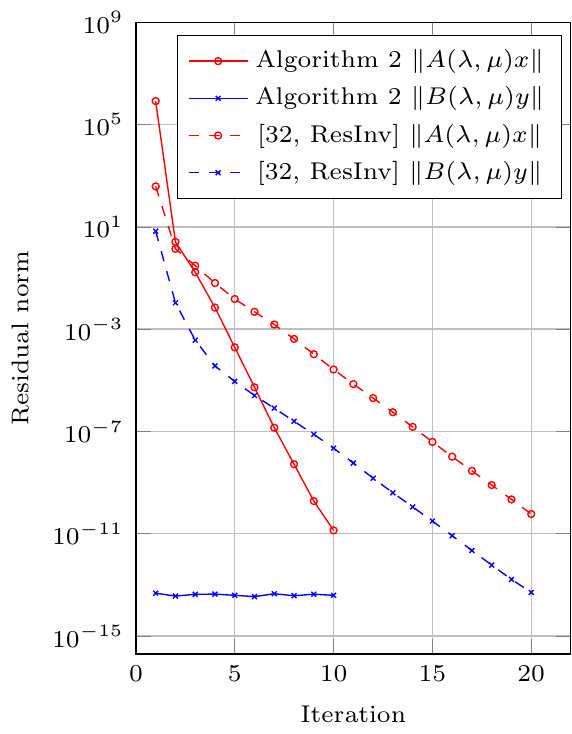}}\;\;\;\;%
   \subfigure[Residual norm vs CPU-time]{\includegraphics{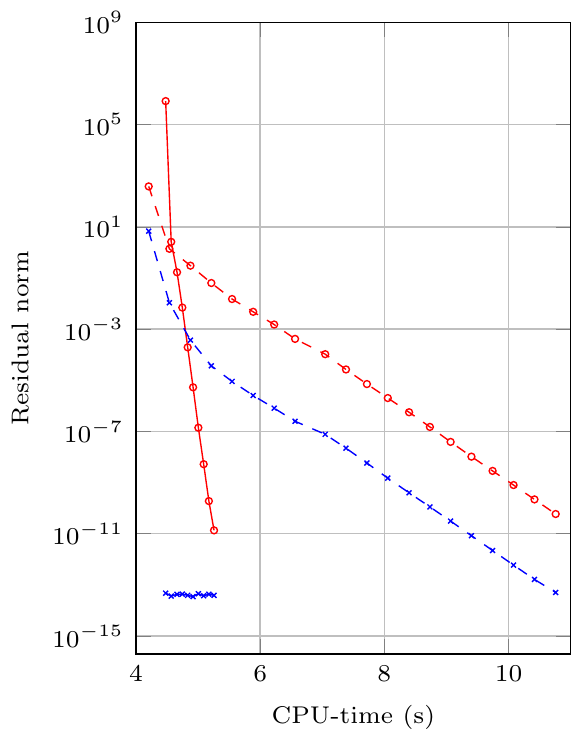}}
   \caption{Visualization of the convergence of Algorithm~\ref{alg:resinv}
and \cite[Algorithm 2 (ResIter)]{Plestenjak:2016:NUMERICAL} for the problem
     in Section~\ref{sec:random}.
      \label{fig:random:resinv}
    }
  \end{center}
\end{figure}

The problem can also be solved with the tensor
infinite Arnoldi method \cite{Jarlebring:2017:TIAR}. More specifically,
we use the implementation of the method
available in the Julia package NEP-PACK
\cite{Jarlebring:2018:NEPPACK} (version 0.2.7).
By directly using Theorem~\ref{thm:der}
we can compute the 60 first derivatives.
The convergence of the first ten eigenvalues
are visualized in Figure~\ref{fig:random:iar}, for
two branches.
The solutions are visualized in Figure~\ref{fig:random:solutions}.

\begin{figure}[h]
  \begin{center}
     \subfigure[Branch 1: $M(\lambda)=A_0+\lambda A_1+g_1(\lambda)A_2$]{\scalebox{1.0}{\includegraphics{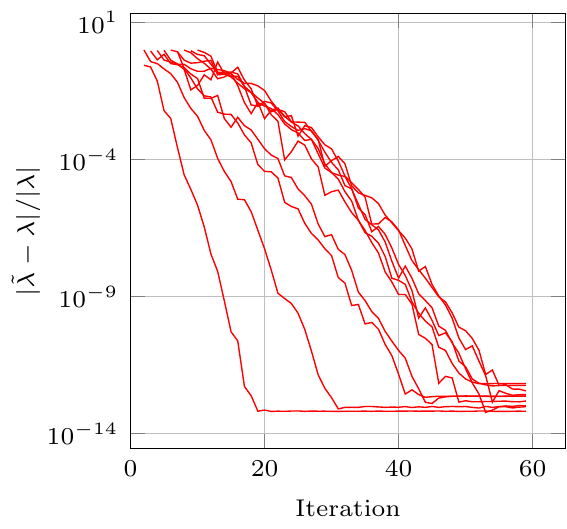}}}\;\;\;\;\;%
   \subfigure[Branch 2: $M(\lambda)=A_0+\lambda A_1+g_2(\lambda)A_2$]{\scalebox{1.0}{\includegraphics{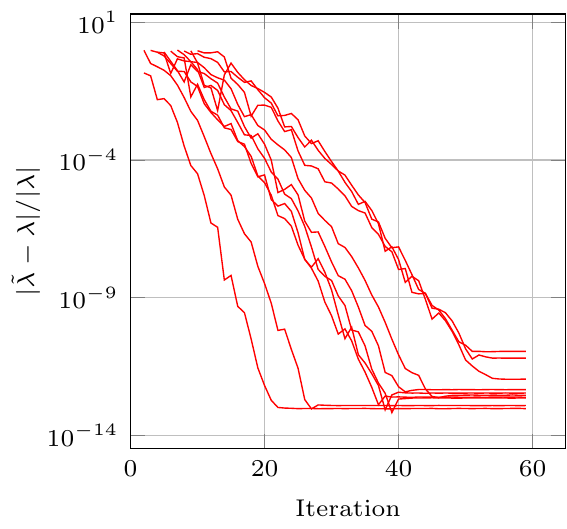}}}%
   \caption{Visualization of convergence of the tensor infinite Arnoldi method for the  problem
     in Section~\ref{sec:random}, for $g_1$ and $g_2$.
      \label{fig:random:iar}
    }
  \end{center}
\end{figure}


\subsection{Domain decomposition example}\label{sec:dd}

      We consider a PDE-eigenvalue problem, which we separate
      into two domains in a way that it leads to a two-parameter
      eigenvalue problem. Although domain decomposition
      (and coupling via boundary conditions) is not new for
      standard eigenvalue problems, the fact that this type of
      domain decomposition can be phrased
      as a two-parameter eigenvalue problem has, to our
      knowledge, not previously been observed. Although the technique seems
      applicable in several physical dimensions, we consider a one-dimensional problem
      for reproducibility.

      Consider the Helmholtz eigenvalue problem defined in the domain $x\in [x_0,x_2]$,
\begin{subequations}\label{eq:pde_dd}
\begin{eqnarray}
  u''(x)+\kappa^2(x)u&=&\lambda u(x)\textrm{ for }x\in [x_0,x_2] \\
  u(x_0)&=&0\\
  u'(x_2)&=&0,
\end{eqnarray}
\end{subequations}
 with a wavenumber $\kappa$ which is discontinuous in one
 part of the domain and smooth in another, as in Figure~\ref{fig:wavenumber}.
 We take a point $x_1$ such that $\kappa$ is smooth for $x>x_1$, assume
 that the solution is non-zero in the interface point $x=x_1$,
 and define
 \[
 \mu:=\frac{u'(x_1)}{u(x_1)} .
 \]
 This means we have two separate PDEs for the two domains:
\begin{subequations}\label{eq:pde_dd1}
 \begin{eqnarray}
   u_1''(x)+\kappa^2(x)u_1&=&\lambda u_1(x),\;\; x_0\le x \le x_1\\
   u_1(x_0)&=&0\\
   u_1'(x_1)-\mu u_1(x_1)&=&0
 \end{eqnarray}
\end{subequations}
  and
\begin{subequations}\label{eq:pde_dd2}
 \begin{eqnarray}
   u_2''(x)+\kappa^2(x)u_2(x)&=&\lambda u_2(x),\;\; x_1\le x \le x_2\\
   u_2'(x_1)-\mu u_2(x_1)&=&0\\
   u_2'(x_2)&=&0.\;\;
 \end{eqnarray}
\end{subequations}
 These are standard linear PDEs (with robin boundary conditions)
 and the uniqueness of these PDEs implies an
 equivalence with the original PDE \eqref{eq:pde_dd}.  See \cite{Lui:2000:DD} and references therein for domain decomposition methods for PDE eigenvalue  problems.

   \begin{figure}[h]
     \begin{center}
      \scalebox{1.0}{\includegraphics{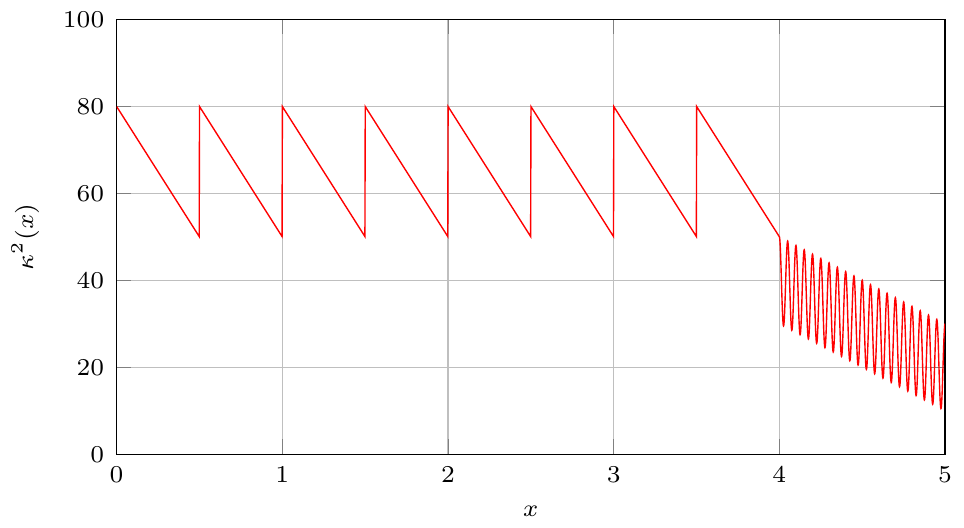}}
       \caption{
         The wavenumber for the example in Section~\ref{sec:dd}. The wavenumber is sinusoidal with
         high frequency in the interval $[4,5]$,
         and discontinuous in $\frac{1}{2},\frac22,\frac{3}{2},\ldots,\frac72$.
         \label{fig:wavenumber}
       }
     \end{center}
   \end{figure}

 The wavenumber is given as in Figure~\ref{fig:wavenumber}, i.e.,
 it is discontinuous at several points in $[x_0,x_1]$
 and with a high frequency decaying sine-curve in $[x_1,x_2]$, representing
 a inhomogeneous periodic medium.
 We invoke different discretizations in the two domains, for the following reasons.
 Since $\kappa$ is discontinuous in $[x_0,x_1]$ spectral discretization in that domain
 will not be considerably faster than a finite difference approximation.
 We therefore use a uniform second order
 finite difference for \eqref{eq:pde_dd1} to obtain sparse matrices and one sided second order finite different scheme for the boundary condition. A spectral discretization is used for $[x_1,x_2]$ where the wavenumber is smooth.
 Since $\mu$ appears linearly in the boundary condition,
 the discretization leads to a two-parameter
   eigenvalue problem of the type \eqref{eq:2ep}. In our setting
   $A_1,A_2,A_3$ are large and sparse, and
   $B_1,B_2,B_3$ are full matrices of smaller size.
   We use the discretization parameters such that $n=10^6$ and $m=30$,
   and $x_0=0$, $x_1=4$ and $x_2=5$.
   In order to make the measurement of error easier, we use left diagonal scaling of the problem
   such that the diagonal elements of
   $A(1.0,1.0)$ and $B(1.0,1.0)$ are equal to one.

   The eigenvalues and some corresponding eigenfunctions are plotted in
   Figure~\ref{fig:example2_solution_values} and Figure~\ref{fig:example2_solutions}.
   The nonlinear function $g_1$ of this problem is
   given in Figure~\ref{fig:example2_function}. Clearly the function has
   singularities for some real $\lambda$-values. The convergence
   of Algorithm~\ref{alg:augnewton} and Algorithm~\ref{alg:resinv}
   are again compared to \cite{Plestenjak:2016:NUMERICAL} in
   Figure~\ref{fig:example2:resinv}. We again conclude that both our
   approaches are competitive, although not always faster
   in terms of iterations, but our approach is generally faster in
   terms of CPU-time.
   The algorithms
   are initiated with approximate rounded eigenvectors and eigenvalues
   close to the eigenvalue $\lambda\approx 18$. We note that our methods
   do not require a starting value for $\mu$ (in contrast to \cite{Plestenjak:2016:NUMERICAL})
   which is an attractive feature from an application point of view,
   since the value $\mu=u'(x_1)/u(x_1)$ is artificially introduced parameter and
   may not be easy to estimate.

   \begin{figure}[h]
     \begin{center}
       \scalebox{1.0}{\includegraphics{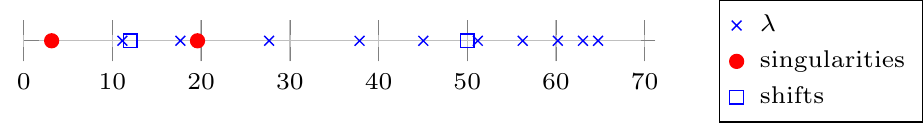}}
       \caption{Computed eigenvalues,
         singularities, and the shifts used in
         the infinite Arnoldi method.
         \label{fig:example2_solution_values}
       }
     \end{center}
   \end{figure}
      \begin{figure}[h]
        \begin{center}
       \scalebox{1.0}{\includegraphics{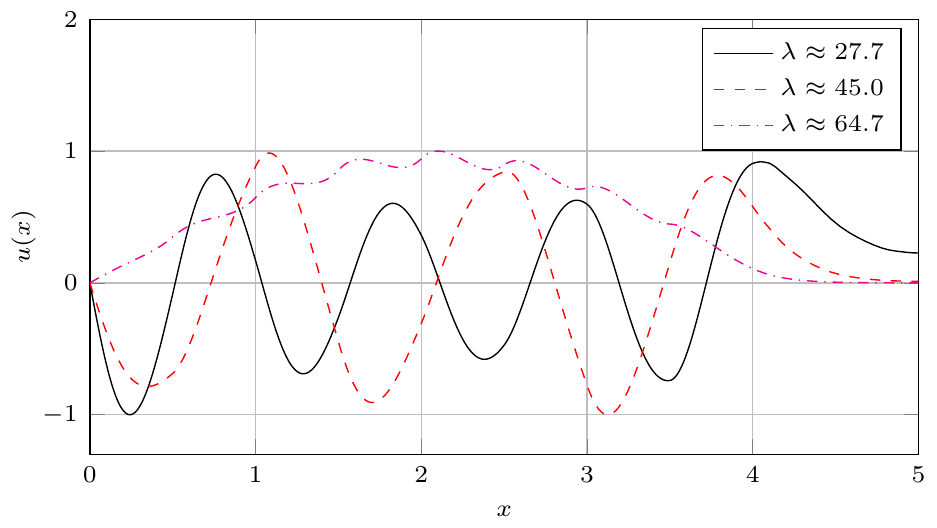}}
       \caption{
Some computed eigenfunctions of the PDE \eqref{eq:pde_dd}
         \label{fig:example2_solutions}
       }
     \end{center}
   \end{figure}
   \begin{figure}[h]
     \begin{center}
       \scalebox{1.0}{\includegraphics{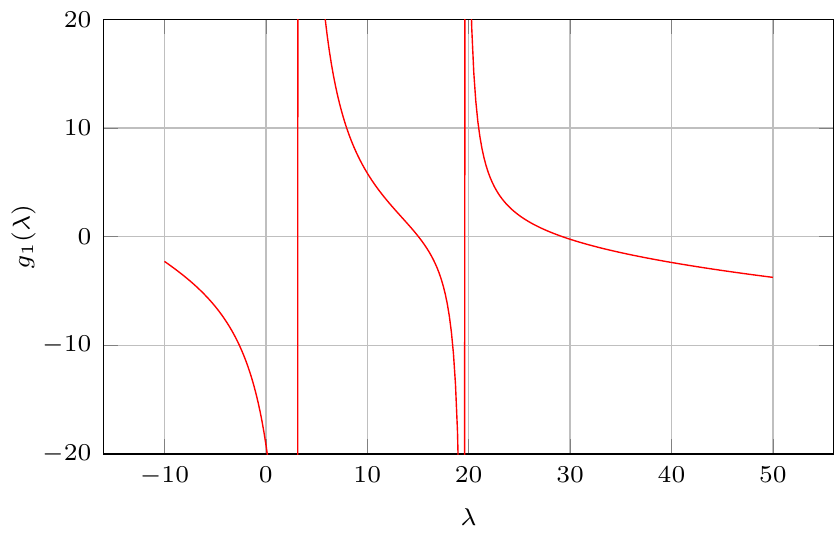}}
       \caption{
The nonlinear function $g_1$ in the simulation in Section~\ref{sec:dd}.
         \label{fig:example2_function}
       }
     \end{center}
   \end{figure}

   We apply the tensor infinite
   Arnoldi method also for this problem. Since this family of methods
   is based on a power series expansion, one can only expect to be able to
   compute eigenvalues on the same side of the singularities as the shift.
   We
   therefore run the algorithm several times for different shifts,
   and select the shifts far away from the singularities,
   as described in Figure~\ref{fig:example2_solution_values}.
   The convergence of the two runs are illustrated in
   Figure~\ref{fig:example2_iar}. Note that the convergence corresponding to one
   eigenvalue for the shift $\sigma=12$ stagnates. This is because
   the eigenvalue is close to the singularity, and therefore difficult to compute, as
   can be seen in Figure~\ref{fig:example2_solution_values}.

   \begin{figure}[h]
     \begin{center}
       \subfigure[Shift $\sigma=12.0$]{\scalebox{1.0}{\includegraphics{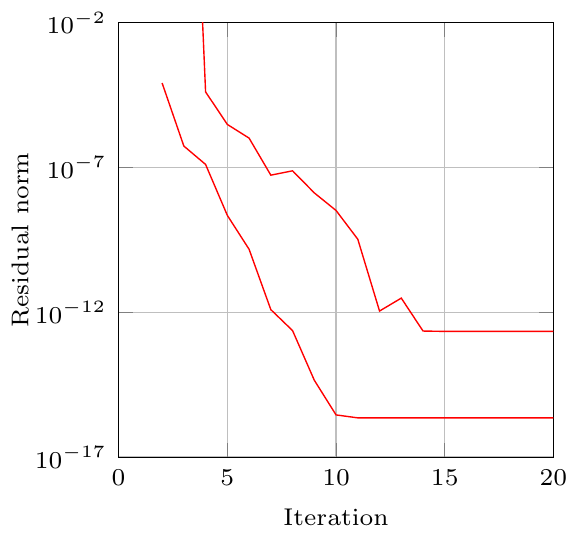}}}
       \subfigure[Shift $\sigma=50.0$]{\scalebox{1.0}{\includegraphics{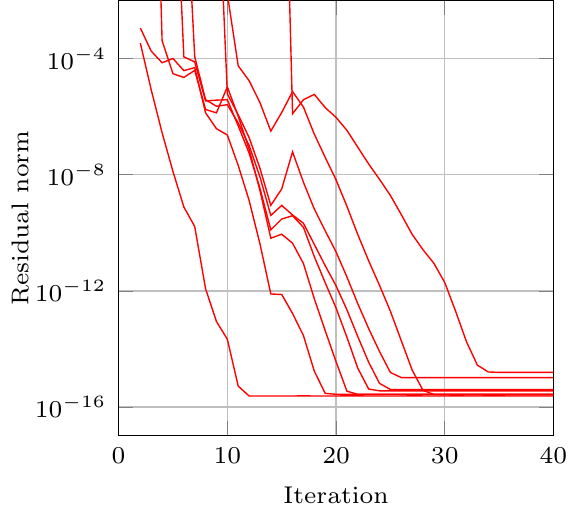}}}
       \caption{
         Convergence history for two different shifts
        \label{fig:example2_iar}
       }
     \end{center}
   \end{figure}

\begin{figure}[h]
  \begin{center}
    \subfigure[Newton-type algorithms]{\includegraphics{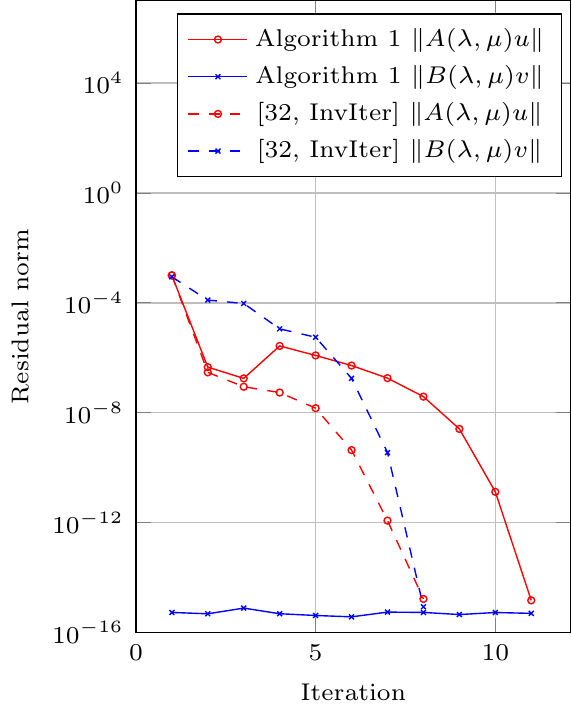}}\;\;\;\;%
    \subfigure[Fixed shift methods]{\includegraphics{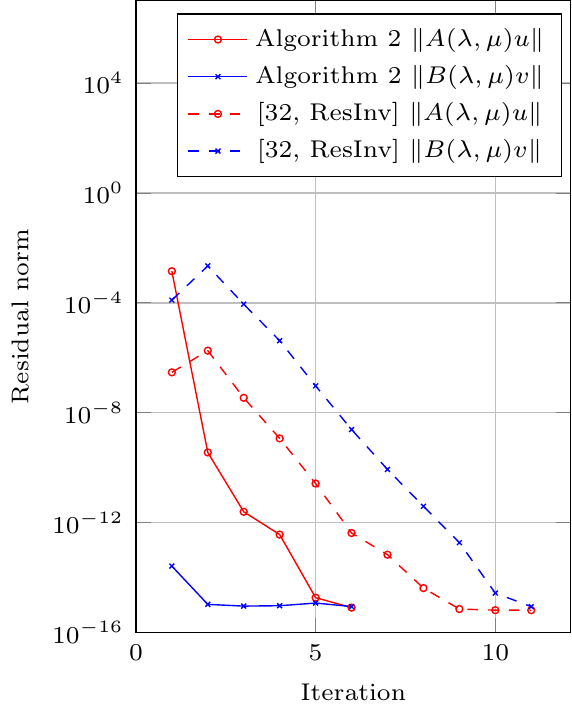}}
   \caption{Visualization of the convergence of
     the proposed algorithms and two
     algorithms in \cite{Plestenjak:2016:NUMERICAL}
     applied to the domain decomposition example in Section~\ref{sec:dd}.
      \label{fig:example2:resinv}
    }
  \end{center}
\end{figure}
\section{Conclusions and outlook}\label{sec:conclusions}
We have presented a general framework
to approach two-parameter eigenvalue problems,
by nonlinearization to NEPs. Several steps
in this technique seem
to be generalizable (but beyond the scope
of the paper), e.g., to general multiparameter
eigenvalue problems essentially by successive
application of the elimination. One such elimination
leads to a nonlinear two-parameter eigenvalue problem
as considered, e.g., in \cite{Plestenjak:2016:NUMERICAL}.

Our paper uses the assumption $m\ll n$ and that
$A_1$, $A_2$ and $A_3$ are large and sparse. We made
this assumption mostly for convenience,
since this allows us to apply
a general purpose method for the parameterized eigenvalue
problem \eqref{eq:B_eig}. If, on the other hand, we wish to
solve two-parameter eigenvalue problems
where these assumptions are not satisfied, the ideas may
still be useful. The GEP \eqref{eq:B_eig} may for
instance be approached with structured algorithms (exploiting
sparsity,
low-rank properties and symmetry), or
iterative solution methods for the GEP, where
early termination is coupled with the NEP-solver.

The generated nonlinear functions $g_i$ are algebraic functions, and
can therefore contain singularities (e.g. branch point singularities as
characterized in Section~\ref{sec:theory}). These can be
problematic in the numerical method, and therefore it would
useful with transformations that remove singularities. Linearization
which do not lead to singularities have been established for
rational eigenvalue problems \cite{Su2011}.

The problem in Section~\ref{sec:dd} is
such that we obtain one large and sparse parameterized matrix $A(\lambda,\mu)$
which
is coupled with a small and dense system.
The setting matches the assumptions of the paper
and is  a representative
example of cases where the behavior is different
in the two physical domains. The example may be generalizable,
to other coupled physical systems where the modeling
in one domain leads to a much smaller matrix,
e.g., using domain decomposition with more physical dimensions.
Note however that the presented methods seem mostly computationally
attractive if the discretization of one domain is much smaller.
If we apply the same technique to domains of equal size,
other generic two-parameter eigenvalue methods (such as
those in \cite{Plestenjak:2016:NUMERICAL}) may be more effective.

\section*{Acknowledgment}
We thank Olof Runborg, KTH Royal Institute of Technology, for discussions regarding the Helmholtz  eigenvalue problem.

\bibliographystyle{plain}
\bibliography{fulljabref}

\end{document}